\documentclass[12pt,a4paper,notitlepage]{amsart}
\usepackage[pdftex]{hyperref}
\usepackage[all]{xy}
\newtheorem{Proposition}{Proposition}[section]
\newtheorem{Theorem}[Proposition]{Theorem}

\newtheorem{Corollary}[Proposition]{Corollary}
\newtheorem{Example}[Proposition]{Example}
\newtheorem{Definition}[Proposition]{Definition}

\theoremstyle{remark}

 \newcommand{\amod}{\A\text{\bf -mod}}
 \newcommand{\amoda}{\A\text{\bf -mod-}\A}
 \newcommand{\amodb}{\A\text{\bf -mod-}\B}

 \newcommand{\Bmod}{\B\text{\bf -mod}}
 \newcommand{\bmoda}{\B\text{\bf -mod-}\A}
 \newcommand{\bmodb}{\B\text{\bf -mod-}\B}
 \newcommand{\bimod}{\text{\bf -mod-}}
 \newcommand{\rmod}{\text{\bf mod-}}
  \newcommand{\lmod}{\text{\bf -mod}}

 \newcommand{\B}{\mathcal{B}}

\renewcommand{\d}{\delta}

 \newcommand{\moda}{\text{\bf mod-}\A}
 
 \newcommand{\ot}{\otimes}
 \newcommand{\ota}{\otimes_{\substack{\A}}}
 \newcommand{\otb}{\otimes_{\substack{\B}}}
 \newcommand{\pota}{\widehat\otimes_{\substack{\A\\{\;}}}}
 \newcommand{\potb}{\widehat\otimes_{\substack{\B}}}
 \newcommand{\pot}{\widehat\otimes}

\newcommand{\A}{\mathcal A}  %

\newcommand{\BB}{\mathbb B}  %
\newcommand {\CC}{\mathbb C} %
\newcommand {\RR}{\mathbb R} %
\newcommand {\NN}{\mathbb N} %
\newcommand {\ZZ}{\mathbb Z} %

\newcommand{\0}{{\emptyset}}
\newcommand{\nul}{\mathbf{o}}
\newcommand{\x}{{\times}}

\renewcommand{\o}{\ifmmode\mathbin{\otimes}\else\char"1C\fi}%

\newcommand{\from}{\leftarrow}
\renewcommand{\H}{\mathcal H}  %

\begin{document}

\title{(Co-)homology of Rees semigroup algebras}
\author{Fr\'ed\'eric Gourdeau}

\address{D\'epartement de math\'ematiques et de statistique,
1045, avenue de la M\'edecine, Universit\'e Laval,  Qu\'ebec (Qu\'ebec),
Canada G1V 0A6}

\email{Frederic.Gourdeau@mat.ulaval.ca}
\author{Niels Gr\o nb\ae k}
\address{Department of Mathematical Sciences, University of
  Copenhagen, DK-2100 Copenhagen \O, Denmark}
\email{gronbaek@math.ku.dk}

\author{Michael C. White}

\address{Department of Mathematics, University of Newcastle, Newcastle upon Tyne, NE1 7RU, England}

\email{Michael.White@ncl.ac.uk}
  \date{\today}
  \begin{abstract}
    Let $S$ be a Rees semigroup, and let $\ell^1(S)$ be its
    convolution semigroup algebra. Using Morita equivalence we show
    that bounded Hochschild homology and cohomology of $\ell^1(S)$ is
    isomorphic to those of the underlying discrete group algebra.
  \end{abstract}


  \maketitle

\section{Introduction}

In this paper we calculate the simplicial cohomology of the
$\ell^1$-algebra of Rees semigroups, motivated by the explicit
computations of the first~\cite{[BDu]} and second order cohomology
groups~\cite{[DaDu]} of several Banach algebras.  These papers contain
a number of results showing that the simplicial cohomology (in
dimensions $1$ and $2$ respectively) of $\ell^1$-semigroup algebras
are trivial for many of the fundamental examples of semigroups.  The
first paper considers the first order simplicial and cyclic cohomology
of Rees semigroup algebras, the bicyclic semigroup algebra and the
free semigroup algebra.  The second paper \cite{[DaDu]} shows that the
second simplicial cohomology vanishes for the semigroups $\ZZ_+$,
semilattice semigroups and Clifford semigroups.  It should be noted
that in each of these papers the arguments are mostly ad hoc.  These
two papers are followed by papers involving some of the current
authors \cite{[GW01]} and \cite{[GPW]} which cover the case of the
third cohomology groups.  In these papers there is some attempt at
systematic methods which might be adapted to cover the case of all
higher cohomology groups, but the authors were not able to do this at
that time.  Finally, later papers \cite{[GJW]}, \cite{[GLW05]},
\cite{[GLW06]} show how to calculate the higher order simplicial
cohomology groups of some of these algebras.  These later papers do
not use ad hoc calculations, but general homological machinery, such
as the Connes-Tsygan long exact sequence and topological simplicial
homology to deduce their results.  The present paper belongs to the
latter family of papers, in that it uses general homological tools.

The proofs in this paper are based firmly on the Morita equivalence methods
developed by the second named author~\cite{[Gr95]} and \cite{[Gr96]}
which apply to a general Banach algebra context. We shall utilize that for the class of
so-called self-induced  Banach algebras, bounded Hochschild
(co-)homology with appropriate coefficients is  Morita invariant.

We briefly describe our general approach. For a semigroup, $T$, with an
absorbing zero, $\0$, there are two Banach algebras that naturally
arize, the discrete convolution algebra $\ell^1(T)$, and the reduced
algebra $\A(T)$ (to be defined below) associated with the inclusion
$\0\hookrightarrow T$. In the case of a Rees semigroup, $S$, with
underlying group $G$, we prove that $\A(S)$ is Morita equivalent to
$\ell^1(G)$.  This may be interpreted as a manifestation of a basic
fact from algebraic topology that for a path connected space $X$ the
fundamental groupoid $\pi(X)$ and the based homotopy group
$\pi_1(X,a)$ are equivalent categories.

Our results then hinge on Theorem 3.1 in which a natural
isomorphism between  the (co-)homology of $\A(S)$ and of the underlying group
algebra is established by means of Morita equivalence. Using an
excision result from \cite{[LW]} we prove that the (co-)\nolinebreak homology of
$\ell^1(S)$ and of $\A(S)$ are isomorphic for a large class of coefficient
modules. Since $\ell^1(S)$ and $\A(S)$ both are H-unital
(cf. \cite{[Wod]}), we further obtain (co-)homology results for the
forced unitizations $\ell^1(S)^\sharp$ and $\A(S)^\sharp$. From this we
derive our main result  that for a Rees semigroup, $S$,
the simplicial cohomology $\H^n(\ell^1(S),\ell^1(S)^*)$ ($n\ge 1$)
of the semigroup algebra is isomorphic to the simplicial
cohomology  $\H^n(\ell^1(G),\ell^1(G)^*)\;(n\geq1)$ of the underlying group algebra.

\section{Basics}

A {\em completely 0-simple semigroup} is a semigroup which has a $0$, has
no proper ideal other than $\{ 0 \}$, and has a primitive
idempotent, that is a minimal idempotent $e$ in the set of non-zero
idempotents (or, precisely, an idempotent $e\ne 0$ such that if
$ef=fe=f\ne 0$ for an idempotent $f$, then $e=f$). Such semigroups
arise naturally and notably in classification theory as quotients of
ideals. It is an important result that any completely 0-simple
semigroup is isomorphic to what we call a Rees semigroup (and
conversely). For this result and more background, see \cite[Chapter 3
and Theorem 3.2.3]{[Ho]}. We now give the definition of a Rees semigroup
and some special cases of Rees semigroups to illustrate how many
natural semigroups are of this form.

The data which are required to define a Rees semigroup are given by two
index sets $I$ and $\Lambda$, a group $G$ and a {\em sandwich matrix}
$P$.  We introduce two zero elements: the first, $\nul$, is an absorbing
element adjoined to $G$ to make the semigroup denoted $G^\nul$.  The
second element, $\0$, is an absorbing zero for the Rees semigroup
itself.  The sandwich matrix $P = (p_{\lambda i})$, is a set of
elements of $G^\nul$ indexed by $\Lambda\x I$, such that each row and
column of $P$ has at least one non-zero entry.  The Rees semigroup
$S$ is then the set $I\x G \x \Lambda \cup \{\0\}$, where $\0$ is
an absorbing zero for the semigroup, and the other products are
defined by the rule
\begin{eqnarray*}
(i, g, \lambda) (j, h, \mu) =
 \begin{cases}
  (i, g p_{\lambda j} h, \mu) & (p_{\lambda j} \not= \nul) \\
  \0   & (p_{\lambda j} = \nul). \\
 \end{cases}
 \end{eqnarray*}

\subsection{Examples}

There are two extreme, degenerate cases which provide good intuition for the logic of calculations.

The first is the case where $I$ and $\Lambda$ are just singletons
and the sandwich matrix consists of the identity of $G$.
In this case $S$ is just $G^\nul$. %
The reduced semigroup algebra, $\A(S)$ defined below, will give us
$\ell^1(G)$ in this case.

The second case has the group, $G$, being trivial and the index sets
being both equal to the set of the first $n$ natural numbers $\{1, 2,
\ldots, n\}$. The sandwich matrix is diagonal with the group identity
repeated along the diagonal. If we identify $G$ with
$\{1\}\subseteq\CC$ the Rees semigroup becomes the system of matrix
units together with the zero-matrix, $S=\{E_{ij}: 1\leq i,j \le
n\}\cup\{0\}$, and $\ell^1(S)$ is the algebra of complex $2$-block
$(n+1)\times(n+1)$ matrices consisting of an upper $n\times n$ block
and a lower $1\times1$ block.  The reduced semigroup algebra, $\A(S)$,
does not have this deficiency and is exactly the matrix algebra,
$M_n(\CC)$.

Our third example is from homotopy theory and is almost generic for
the concept of Rees semigroups. Recall that a small category in which
every morphism is invertible is called a \emph{groupoid}. It is
\emph{connected} if there is a morphism between each pairs of
objects. The canonical example is $\pi(X)$ the \emph{fundamental
  groupoid} of a path connected topological space $X$. The objects are
the points of $X$ and for $x,y\in X$ the morphism set $\pi(X)(x,y)$ is
the set  of homotopy classes relative to $x,y$ of paths from $x$ to
$y$ with composition derived from products of paths. In particular $\pi(X)(x,x)=\pi_1(X,x)$ for each $x\in X$, so that the
fundamental group of $X$ at $x$ is identified with the full
subcategory of $\pi(X)$ with just one object $x$. Since $X$ is path connected
$\pi_1(X,x)\cong\pi_1(X,y)$ for each $x,y\in X$, for details see \cite{[RB]} 

Fix $a\in X$ and
set $G=\pi_1(X,a)$. We associate Rees semigroups to $\pi(X)$ in the following way. For
each $x\in X$ choose $s_x\in \pi(X)(a,x)$ and
$t_x\in\pi(X)(x,a)$. This gives bijections
$\psi_{x,y}\colon\pi(X)(x,y)\to G$
$$
\psi_{x,y}(\gamma)=s_x\gamma t_y,\quad x,y\in X,\;\gamma\in
\pi(X)(x,y).
$$  
Then we get
\[
\psi_{x,z}(\gamma\gamma')=\psi_{x,y}(\gamma)(s_yt_y)^{-1}\psi_{y,z}(\gamma'),\quad
x,y,z\in X,\;
\gamma\in\pi(X)(x,y),\;\gamma'\in\pi(X)(y,z).
\]
Let $I,\Lambda$ be sets and consider maps $\alpha\colon I\to
X,\;\beta\colon\Lambda\to X$. We define multiplication on $(I\times
G\times\Lambda)\cup\0$ as follows. For $i,j\in I,\;g,h\in
G,\;\lambda,\mu\in\Lambda,$ put
$\gamma=\psi_{\alpha(i),\beta(\lambda)}^{-1}(g),\;\gamma'=\psi_{\alpha(j),\beta(\mu)}^{-1}(h)$
and set
\begin{eqnarray*}
(i, g, \lambda) (j, h, \mu) =
 \begin{cases}
  (i,\psi_{\alpha(i),\beta(\mu)}(\gamma\gamma'), \mu) &(\alpha(j)=\beta(\lambda)) \\
  \0   & (\alpha(j)\neq\beta(\lambda))\\
 \end{cases}
 \end{eqnarray*}
so that the product is simply given by products of paths, when
defined. The sandwich matrix is 
\begin{eqnarray*}
p_{\lambda i}=\begin{cases} (s_{\alpha(i)}t_{\beta(\lambda)})^{-1}&(\alpha(i)=\beta(\lambda))\\
\nul&(\alpha(i)\neq\beta(\lambda)).
\end{cases}
\end{eqnarray*}

The condition that the sandwich matrix $(p_{\lambda i})$ has a
non-zero entry in each row and each column is
$\alpha(I)=\beta(\Lambda)$.

It is a simple fact, but crucial to the use of groupoids in algebraic
topology cf. \cite[Chp. 8]{[RB]}, that the fundamental groupoid $\pi(X)$ and the
based homotopy group $\pi_1(X,a)$ are equivalent categories. Our main
result on Morita equivalence is a manifestation of this fact in the
setting of convolution Banach algebras.

\subsection{Background on homological algebra}

The main result of this paper concerns boun\-ded (co-)homology of Banach algebras,
so we will give a brief description of the theory of
(co-)\nolinebreak homology of Banach algebras, as we fix the notation we will use
for this paper. For further details we refer to \cite{[He89]}.

For a Banach algebra $\A$ we denote the categories of left (right)
Banach $\A$-modules and bounded module homomorphisms by $\amod$
(respectively $\moda$). If $\B$ is also a Banach algebra, the category
of Banach $\A-\B$ bimodules and bounded homomorphisms is $\amodb$. A
{\em full subcategory} is a subcategory $\frak C$ which includes all
morphisms between objects of $\frak C$.

Let $\A$ be a Banach algebra, $\A^\#$ be its
  forced unitization and $X\in \amod$. The \emph{bar
    resolution} of $X$ is the complex
\[
\BB(\A,X)\colon 0\from X\from \A^\#\pot X\from \A^\#\pot\A\pot
X\from\cdots\from\A^\#\pot\A\pot\cdots\pot\A\pot X\from\cdots
\]
with boundary maps
\[
b(a_1\ot\cdots\ot a_n\ot x)=\sum_{k=1}^{n-1}(-1)^{k-1}a_1\ot \cdots \ot
a_ka_{k+1}\ot\cdots \ot x+(-1)^{n-1}a_1\ot\cdots\ot a_{n-1}\ot a_nx
\]
Similarly we define the bar resolution for $X\in\moda$. It
is a standard fact that $\BB(\A,X)$ is contractible.

The \emph{simplicial complex} of $\A$ is the subcomplex of
$\BB(\A,\A)$ with the first tensor factor in $\A$ rather than $\A^\#$:
\[
0\from\A\from\A\pot\A\from\cdots\from \A^{n\pot}\from \cdots\from
\]

If the dual complex of the simplicial complex
\[
0\to\A^*\to(\A\pot\A)^*\to\dots\to (A^{n\pot})^* \to \cdots \to
\]
splits as a complex of Banach spaces, then $\A$ is \emph{H-unital} \cite{[Wod]}.

A module $X\in\amodb$ is \emph{induced} if the multiplication
$$
\A\pota X\potb\B\to X\colon a\ota x\otb b\mapsto axb
$$
is an isomorphism. If $\A$ is induced as a module in $\amoda$ then
$\A$ is \emph{self-induced}.

A bounded linear map $L\colon E\to F$ between Banach spaces is
\emph{admissible} if $\ker E$ and $\operatorname{im} E$ are
complemented as Banach spaces in $E$ respectively $F$.

A module $P\in\amod$ is  \emph{(left) projective} if, for every admissible
epimorphism $q\colon Y\to Z$, all lifting
problems in $\amod$
\[
\xymatrix{
{}&P\ar[d]^{\phi}\ar@{-->}[dl]|?&{}\\
Y\ar[r]_{q}&Z\ar[r]&0
}
\]
can be solved. If all, not just admissible, lifting problems can be
solved, then $P$ is \emph{strictly} projective. The module $P$ is
\emph{(left) flat} if for every admissible short exact sequence in $\moda$
$$
0\to X\to Y\to Z\to 0
$$
the sequence
$$
0\to X\pota P\to Y\pota P\to Z\pota P\to0
$$
is exact, and \emph{strictly} flat if the requirement of admissibility
can be omitted.

The fundamental concept of our approach is Morita equivalence.

\begin{Definition} Two self-induced Banach
  algebras $\A$ and $\B$ are \emph{Morita equivalent} if there are induced
  modules $P\in\bmoda$ and $Q\in\amodb$ so that
\[P\pota Q\cong\B \quad\text{and}\quad Q\potb P\cong \A,\]
where the isomorphisms are implemented by bounded bilinear balanced module maps
$[\cdot,\cdot]\colon P\times Q\to\B$ and $(\cdot,\cdot)\colon
Q\times P\to \A$ satisfying
\[
[p,q].p'=p.(q,p')\quad q.[p,q']=(q,p).q'\quad p,p'\in P;\; q,q'\in Q.
\]
\end{Definition}

Our objective is to describe bounded Hochschild (co-)homology of Rees
semigroup algebras in terms of the (co-)homology of the algebra of the underlying
group. First we define homology.

\begin{Definition}
For $X\in\amoda$ the Hochschild complex is
\[
0\from X\from X\pot \A\from\cdots \from X\pot\A^{n\pot}\from\cdots
\]
with boundary maps given as
\begin{eqnarray*}
\d(x\ot a_1\ot\cdots\ot a_n)&=&xa_1\ot
a_2\ot\cdots \ot a_n
+x\ot\sum_{k=1}^{n-1}(-1)^k a_1 \ot\cdots\ot
a_ka_{k+1}\ot\cdots \ot a_n\\
& & +(-1)^n a_nx\ot a_1\ot\cdots\ot a_{n-1}
\end{eqnarray*}
The \emph{bounded Hochschild homology} of $\A$ with coefficients in
$X$, $\H_n(A,X),\; n=0, 1, \dots$, is the homology of this complex.
The \emph{bounded Hochschild cohomology} of $\A$ with coefficients in
the dual module $X^*$, $\H^n(A,X^*),\; n=0, 1, \dots$, is the
cohomology of the dual complex
\[
0\to X^*\to(\A\pot X)^*\to\cdots\to(\A^{n\pot}\pot X)^*\to\cdots
\]
\end{Definition}

Our main result hinges on the fact that bounded Hochschild homology
and cohomology under certain conditions are Morita invariant, cf \cite{[Gr96]}. We state
the version that we need in the paper.

\begin{Theorem}\label{invariance} Let $\A$ and $\B$ be Morita equivalent Banach
  algebras with  implementing modules $P\in\bmoda$ and $Q\in\amodb$. If
  $P$ is
  right flat as a module in $\moda$ and left flat as
  a module in $\Bmod$, then there are natural isomorphisms
\[
\H_n(\A,X)\cong\H_n(\B,P\pota X\pota Q)\text{ and }\H^n(\A,X^*)\cong\H^n(\B,(P\pota X\pota Q)^*)
\]
for all induced modules $X\in \amoda$.
\end{Theorem}
\begin{proof} To establish the homology statement we briefly recall the Waldhausen first quadrant double
  complex. For details we refer to  \cite[pp. 132-133]{[Gr96]}.
On the horizontal axis it is the Hochshild  complex  in $\amoda$ for $X$
and on the vertical axis it is the Hochschild  complex in $\bmodb$ for
$P\pota X\pota Q$. For $n\geq1$ the $n$'th row is
$$
P\pota \BB(\A,X\pota Q\pot\B^{\pot(n-1)})
$$
and the $n$'th column is
$$
\BB(\B,\A^{\pot(n-1)}\pot X\pota Q)\potb P.
$$
By \cite[Lemma 6.1]{[Gr95]} it suffices that rows and
columns are acyclic for $n\geq1$. As the bar complexes are contractible,
this follows from the flatness properties of $P$.

Dualizing the Waldhausen double complex we obtain the cohomology
statement.
\end{proof}

\subsection{Semigroup algebras}

Given a semigroup $T$, the semigroup algebra is the Banach space
$\ell^1(T)$, equipped with the product which extends the product
defined on the natural basis from $T$, by bilinearity, to the whole of
$\ell^1(T)$. Throughout, we denote an element of the natural
basis for $\ell^1(T)$, corresponding to $t \in T$, by $t$ itself.

An \emph{absorbing element}  for a semigroup $T$ is an element $\0\in
T$ so that $t\0=\0 t=\0$ for all $t\in T$. Obviously there is at most
1 absorbing element in $T$. If $\0$ is absorbing, then $\CC\0$ is a
1-dimensional 2-sided ideal of $\ell^1(T)$. Our calculations are more
easily done modulo this ideal.

\begin{Definition} Let $T$ be a semigroup with  absorbing element
  $\0$. The \emph{reduced semigroup algebra} is
\[
\A(T)=\ell^1(T)\big/\CC\0.
\]
As a Banach space, $\A(T)$ is isometrically isomorphic to
$\ell^1(T\setminus\{\0\})$ and the multiplication is given by
\begin{equation*}
st=\left\{\begin{array}{rl} st&\text{if } st\neq\0\text{ in }T,\\
                         0&\text{if } st=\0\text{ in }T.\end{array}\right.
\end{equation*}

We note that if the semigroup satisfies $T^2=T$
then the multiplication maps
$$
\ell^1(T)\pot\ell^1(T)\to\ell^1(T)\text{ and }\A(T)\pot\A(T)\to\A(T)\label{surj}
$$
are both surjective.

For $X\in\ell^1(T)\bimod\ell^1(T)$ the \emph{reduced module} is
\[
\widetilde X=\frac{X}{\0 X + X \0}.
\]
The reduced module is canonically a module in $\A(T)\bimod\A(T)$.
\end{Definition}

\begin{Example} The concept of reduced semigroup algebra fits in the more
  general context of extension of semigroups. If $I$ is a semigroup
  ideal of a semigroup $T$, then the Rees factor semigroup $T/I$ has the equivalence class $I$ as an
  absorbing zero and we get the corresponding admissible extension of Banach
  algebras
$$
0\to\ell^1(I)\to\ell^1(T)\to\A(T/I)\to0.
$$
\end{Example}

Note that every module in $X\in\A(T)\bimod\A(T)$ can be obtained as a reduced
module from a module in $\ell^1(T)\bimod\ell^1(T)$, simply by extending to an
action of $T$ by $\0X=X\0=\{0\}$, so that in this case $\widetilde X=X$.

From \cite {[LW]}, we get the following proposition relating the Hochschild homology and cohomology of $\ell^1(T)$ and $\A(T)$.
\begin{Proposition}\label{reduced} Let  $X\in\ell^1(T)\bimod\ell^1(T)$ be such $\0 X = X \0$. Then
$$
\H_n(\ell^1(T),X)\cong\H_n(\A(T),\widetilde X)\text{ and }\H^n(\ell^1(T),X^*)\cong\H^n(\A(T),\widetilde X^*)
$$
for $n=0,1,\dots$
\end{Proposition}
\begin{proof} By  \cite [Theorem 4.5]{[LW]} we have the long exact sequence
$$
\cdots\to\H_n(\CC\0,\0 X)\to\H_n(\ell^1(T),X)\to\H_n(\A(T),\widetilde
X)\to\H_{n+1}(\CC\0,\0 X)\to\cdots.
$$
As $\CC\0\cong\CC$ we have $\H_n(\CC\0,\0 X)=\{0\}$ for all $n\geq0$,
yielding the claim. A similar application of  \cite [Theorem 4.5]{[LW]} to cohomology
gives the other statement.
\end{proof}

As a consequence, since our concern is to determine Hochschild
homology and cohomology, we shall work with reduced semigroup algebras
in the following.

\subsection{The Rees semigroup algebra}
For a Rees semigroup with index sets $I$ and $\Lambda$ over a group
$G$, set
$$
_iS_\lambda=\{i\}\times G\times\{\lambda\}\quad i\in
I,\,\lambda\in\Lambda.
$$
Then
$$
\A(S)=\bigoplus_{\substack{i\in I\\\lambda\in\Lambda}}\ell^1(
_iS_\lambda)
$$
is a decomposition of $\A(S)$ into an $\ell^1$-direct sum of
subalgebras such that
\begin{align*}
&\ell^1( _iS_\lambda)\cong\ell^1(G),\text{ if }p_{\lambda i}\neq\nul\\
&\ell^1( _iS_\lambda)\cdot\ell^1( _jS_\mu)=\{0\},\text{ if
}p_{\lambda j}=\nul.
\end{align*}
The isomorphism above is implemented by the semigroup
isomorphism
\[G\to _i\nolinebreak\! S_\lambda\colon g\mapsto (i,gp_{\lambda
  i}^{-1},\lambda).
  \]

Since $\ell^1( _iS_\lambda)\cdot\ell^1( _jS_\mu)\subseteq\ell^1(
_iS_\mu)$ this decomposition is organized as a rectangular band. We
further put
\begin{align*}
& _iS=\{i\}\times G\times\Lambda,\quad i\in I\\
& S_\lambda=I\times G\times\{\lambda\},\quad\lambda\in\Lambda
\end{align*}
so that $\ell^1( _iS)=\bigoplus_{\lambda\in\Lambda}\ell^1(
_iS_\lambda)$ for each $i\in I$ is a closed right
ideal of $\A(S)$ and  $\ell^1(S_\lambda)=\bigoplus_{i\in I}\ell^1(
_iS_\lambda)$ for each $\lambda\in\Lambda$ is a closed left
ideal of $\A(S)$.

For the remainder of this section we fix a Rees semigroup and
establish a number of generic properties.

\begin{Proposition} Each $\ell^1( _iS),\;i\in I$, is a left unital
  right ideal and each $\ell^1(S_\lambda),\;\lambda\in\Lambda$, is a
  right unital left ideal.
\end{Proposition}

\begin{proof}
 Given an element of the indexing set $i \in I$ there is, by the property of the sandwich
matrix $P$,  a nonzero entry $p_{\mu i}$, for some $\mu\in
\Lambda$, (as each row and column has a non-zero entry).  We define
$e_i = ( i , p_{\mu i}^{-1}, \mu)$.
The element $e_i$ acts as a left identity for $s\in _i\!\!S$ as
 $$
e_i  s  = ( i , p_{\mu i}^{-1}, \mu) (i, g, \lambda)
     = ( i , p_{\mu i}^{-1} p_{\mu i} g, \lambda) = ( i , g, \lambda)
     = s.
$$
In particular $e_i$ is idempotent, and is clearly a left identity for
$\ell^1( _iS)$. Similarly on the right we have a non-zero element
$p_{\lambda j}$ of $P$, which gives the required element as $(j,
p_{\lambda j}^{-1}, \lambda)$.
 \end{proof}

Note that the idempotent $e_i$ is not necessarily unique as we can
form such an idempotent using any index from $\Lambda$ which gives a
non-zero entry in $P$. However, in what follows it will be useful to
have a fixed family of left and right  idempotents in mind.

\begin{Definition}
We
fix a family of left (and right) idempotents denoted $\{e_i\}_{i\in
I}$ (and $\{f_\lambda\}_{\lambda \in \Lambda}$), which are left
(respectively right) units for $\ell^1(_iS)$
(respectively $\ell^1(S_\lambda)$).
\end{Definition}

\begin{Proposition} \label{projp} For each $i\in I$, the right ideal $\ell^1( _iS)$
  is strictly projective in $\rmod\A(S)$ and, for each
  $\lambda\in\Lambda$, the left ideal $\ell^1(S_\lambda)$ is strictly
  projective in $\A(S)\lmod$.
\end{Proposition}
\begin{proof} We give the proof for $\ell^1( _iS)$ as the other is
  completely analogous. Let $q\colon Y\to Z$ be an epimorphism in
  $\rmod\A(S)$ and consider the lifting problem
\[
\xymatrix{
&\ell^1( _iS)\ar[d]^{\phi}&\\
Y\ar[r]^{q}&Z\ar[r]&0
}
\]
Choose $y_i\in Y$ so that $q(y_i)=\phi(e_i)$ and define for $s\in I\times
G\times\Lambda$
$$
\tilde\phi(e_is)=y_ie_is.
$$
As $_iS=e_i(I\times G\times \Lambda)$ in $\A(S)$ the universal property
of $\ell^1$-spaces provides a bounded linear map $\tilde\phi
\colon\ell^1( _iS)\to Y$ so that $q\circ\tilde\phi=\phi$. Clearly
$\tilde\phi$ is a right module map.
\end{proof}

\begin{Corollary} \label{proja}$\A(S)$ is strictly projective in
  $\rmod\A(S)$ and in $\A(S)\lmod$.
\end{Corollary}
\begin{proof} We utilize the direct sum decomposition
  $\A(S)=\bigoplus_{i\in I}\ell^1( _iS)$ in $\rmod \A(S)$.
  Consider the lifting problem
\[
\xymatrix{
&\ell^1( _iS)\ar[d]^{\kappa_i} \ar[ddl]_{\tilde\phi_i}&\\
&\A(S)\ar[d]^{\phi}\ar@{-->}[dl]&\\
Y\ar[r]^{q}&Z\ar[r]&0
}
\]
where $\kappa_i,\;i\in I$, are the natural inclusions and
$\tilde\phi_i,\;i\in I$, are the lifts of $\phi\circ\kappa_i$
constructed in \ref{projp}. By the open mapping theorem applied to $q$, we can choose
the elements $y_i$ such that $\Vert y_i \Vert \le C$ for some constant $C$, and therefore such that
$\Vert\tilde\phi_i\Vert\leq C$ for all $i\in I$. Thus there is a
unique module map $\tilde\phi\colon\A(S)\to Y$ with
$\tilde\phi\circ\kappa_i=\tilde\phi_i$ for all $i\in I$. Since both
$q\circ\tilde\phi$ and $\phi$ complete the direct sum diagram for the
maps $\phi\circ\kappa_i$, it follows by uniqueness of universal
elements that $q\circ\tilde\phi=\phi$.

The case of left projectivity is completely analogous.
\end{proof}
\begin{Corollary}\label{hunital}
$\A(S)$ and $\ell^1(S)$ are H-unital. In particular $\A(S)$ and
$\ell^1(S)$ are self-induced.
\end{Corollary}
\begin{proof} The statement about $\A(S)$ is a general fact about
  one-sided projective Banach algebras with surjective multiplication
  $\A\pot\A\to\A$. Let $\rho\colon\A\to\A\pot\A$ be a splitting of
  multiplication provided by right projectivity of $\A$. Then
  $\rho\ot\bold1\colon \A^{\pot n}\to \A^{\pot(n+1)}$ is a contracting
  homotopy of the simplicial complex:
\[
\xymatrix{0&\A\ar[l]\ar[d]_{\bold1}\ar[dr]^{\rho}&\A^{\pot 2}\ar[l]^{b}\ar[d]_{\bold1}\ar[dr]^{\rho\ot\bold1}&\cdots\ar[l]^{b}&\\
          0& \A\ar[l]                            &\A^{\pot2}\ar[l]^{b}   &\A^{\pot3}\ar[l]^{b}&\cdots\ar[l]
}
\]
For $a_1\ot\cdots\ot a_n\in\A^{\pot n}$
\begin{align*}
\rho\ot\bold1( b(a_1\ot\cdots\ot
a_n))&=\rho(a_1a_2)\ot a_3 \ot\cdots\ot a_n -\rho(a_1)\ot b(a_2\ot\cdots\ot a_n)\\
& =\rho(a_1)a_2\ot\cdots\ot a_n-\rho(a_1)\ot b(a_2\ot\cdots\ot a_n),\\
\end{align*}
and
\begin{align*}
b(\rho\ot\bold1(a_1\ot\cdots\ot a_n))&=a_1\ot\cdots\ot a_n\\
&\phantom{=}-\rho(a_1)a_2\ot\cdots\ot a_n +\rho(a_1)\ot b(a_2\ot\cdots\ot a_n).
\end{align*}

An equivalent formulation of H-unitality is that $\H_n(\A,X)=\{0\}$
for all trivial modules, i.e. modules with $\A X=X\A=\{0\}$. Since $\widetilde X=X$
for any trivial module the statements about $\ell^1(S)$ follow from \ref{reduced}
\end{proof}

\subsection{Morita equivalence}

Now fix one of the idempotents $e=(i,p_{\lambda i}^{-1},\lambda)$ and
put
$$
P=e\A(S)\text{ and }Q=\A(S)e,
$$
so that $P$ is the closed right ideal $\ell^1( _iS)$ and $Q$ is the
closed left ideal $\ell^1(S_\lambda)$. Let further
$$
\B=e\A(S)e,
$$
so that $\B=\ell^1( _iS_\lambda)\cong\ell^1(G)$.

For brevity in this section, let $\A=\A(S)$. Then $P\in
\B\bimod\A$ and $Q\in \A\bimod\B$. Our main result is

\begin{Theorem} The modules $P$ and $Q$ give a Morita equivalence of
  $\A$ and $\B$, that is, multiplication gives bimodule isomorphisms
\[
\A\cong Q\potb P\text{ and }\B\cong P\pota Q.
\]
\end{Theorem}
\begin{proof} Clearly the multiplication $P\pot Q\to \B$ is
  surjective, see \ref{surj}. Now suppose that
$$
\sum_n e a_n b_n e=0 \text{ for }\sum_n\Vert ea_n\Vert\Vert
b_ne\Vert<\infty,\;a_n,b_n\in\A.
$$
Then
$$
\sum_n ea_n\ota b_ne=\sum_n ea_nb_ne\ota e=0
$$
so that multiplication $P\pota Q\to \B$ is injective. It follows that
$P\pota Q\cong \B$.

For the reversed tensor product first note that
\[
(j,g,\mu)=(j,g,\lambda)(i,p_{\lambda i}^{-1},\mu)
\]
for all $j\in I,g\in G,\mu\in\Lambda$, so that the multiplication
$Q\potb P\to\A$ is surjective.

Identifying $Q\pot P$ with $\ell^1(Se\times eS)$ a generic element in
$Q\pot P$ has the form
$$
\sum_{j,g,h,\mu}\alpha_{jgh\mu}(j,g,\lambda)\ot(i,h,\mu).
$$
Assume that
\[
\sum_{j,g,h,\mu}\alpha_{jgh\mu}(j,g,\lambda)(i,h,\mu)=0.
\]

This means that

\[
\sum_{\substack{g,h\\gp_{\lambda i}h=\gamma}}\alpha_{jgh\mu}=0
\]
for each $j\in I,\gamma\in G,\mu\in\Lambda$. Now
\begin{align*}
(j,g,\lambda)\otb(i,h,\mu)&=(j,g,\lambda)\otb(i,h,\lambda)(i,p_{\lambda
  i}^{-1},\mu)\\
&=(j,g,\lambda)(i,h,\lambda)\otb(i,p_{\lambda
  i}^{-1},\mu)\\
&=(j,gp_{\lambda i}h,\lambda)\otb(i,p_{\lambda
  i}^{-1},\mu),
\end{align*}
so
$$
\sum_{j,g,h,\mu}\alpha_{jgh\mu}(j,g,\lambda)\otb(i,h,\mu)=
   \sum_{\gamma,j,\mu}
      \big[  \sum_{\substack{g,h\\gp_{\lambda i}h=\gamma}}\alpha_{jgh\mu} \big]
    (j,\gamma,\lambda)\otb(i,p_{\lambda i}^{-1},\mu)=0.
$$
It follows that multiplication $Q\potb P\to\A$ is injective so that
$Q\potb P\cong\A$. All together, $\A$ and $\B$ are Morita equivalent.
\end{proof}

We want to establish Morita invariance of Hochschild homology. We have
already noted that $P$ is strictly projective in $\moda$. We now prove

\begin{Theorem}
The module $P=e\A$ is strictly projective in $\Bmod$.
\end{Theorem}
\begin{proof} Consider the direct sum decomposition in $\Bmod$
$P=\bigoplus_\mu \ell^1( _iS_\mu)$. Let
\[
\xymatrix{
&\ell^1( _iS_\mu)\ar[d]^{\phi_\mu}\ar@{-->}[dl]&\\
Y\ar[r]^{q}&Z\ar[r]&0
}
\]
be a lifting problem, where $\phi_\mu$ is the restriction of $\phi:P\to Z$.
From the open mapping theorem for $q$, there exists $y_\mu$ with
$$
\Vert y_\mu\Vert\leq C\text{ and }q(y_\mu)=\phi_\mu((i,p_{\lambda
  i}^{-1},\mu)),
$$
for some constant $C$. Define $\tilde\phi_\mu\colon\ell^1( _iS_\mu)\to Y$ by
$$
\tilde\phi_\mu((i,g,\mu))=(i,g,\lambda)y_\mu.
$$
Then $\tilde\phi_\mu\in\Bmod$ and $\Vert\tilde\phi_\mu\Vert\leq C$. Since
\begin{align*}
q(\tilde\phi_\mu((i,g,\mu))&=q((i,g,\lambda)y_\mu)\\
&=(i,g,\lambda)q(y_\mu)\\
&=(i,g,\lambda)\phi_\mu((i,p_{\lambda i}^{-1},\mu))\\
&=\phi_\mu((i,g,\lambda)(i,p_{\lambda i}^{-1},\mu))\\
&=\phi_\mu((i,g,\mu))
\end{align*}
we have solved the lifting problem. Proceeding as in \ref{proja} we
conclude that $P$ as a direct sum of strictly projective modules is
strictly projective in $\Bmod$.
\end{proof}

\section{Applications to homological properties}

With $P$, $Q$, $\A=\A(S)$, and $\B=e\A(S)e$ as in the previous section
we have functors
\begin{align*}
&\A(S)\bimod \A(S)\to\Bmod\B\colon X\mapsto P\pota X\pota Q,\\
&\B\bimod\B\to\A(S)\bimod\A(S)\colon Y\mapsto Q\potb Y\potb Q.
\end{align*}

Replacing $\B$ by the isomorphic $\ell^1(G)$ we get functors
\begin{align*}
&\Phi\colon\A(S)\bimod\A(S)\to\ell^1(G)\bimod\ell^1(G),\\
&\Gamma\colon\ell^1(G)\bimod\ell^1(G)\to\A(S)\bimod\A(S).
\end{align*}
We collect our findings in
\begin{Theorem}\label{main} The functors $\Phi$ and $\Gamma$ constitute an equivalence of
  the full subcategories of induced bimodules over $\A(S)$ and
  $\ell^1(G)$ and there are natural isomorphisms of (co-)homology functors
\begin{align*}
&\H_n(\A(S),X)\cong\H_n(\ell^1(G),\Phi(X)),\\
&\H^n(\A(S),X^*)\cong\H^n(\ell^1(G),\Phi(X)^*).
\end{align*}
\end{Theorem}
\begin{proof} The equivalence follows from the natural isomorphisms
\begin{align*}
&Q\potb( P\pota X\pota Q)\potb P\cong\A\pota X\pota \A \cong X,\\
&P\pota( Q\potb Y\potb P)\pota Q\cong\B\potb Y\potb \B \cong Y
\end{align*}
for induced modules $X\in \A(S)\bimod\A(S)$ and $Y\in
\B\bimod\B$. As $P$ is strictly projective in $\Bmod$ and in
$\rmod \A(S)$ the statement about (co-)homology groups follows
from \ref{invariance}.
\end{proof}
We recall that $S$ is a Rees semigroup with underlying group $G$. We note a number of consequences.
\begin{Corollary} There are isomorphisms
\begin{align*}
  &\H_n(\A(S),\A(S))\cong\H_n(\ell^1(S),\ell^1(S))\cong\\
 &\H_n(\A(S)^\#,\A(S)^\#)\cong\H_n(\ell^1(S)^\#,\ell^1(S)^\#)\cong\\
 &\H_n(\ell^1(G),\ell^1(G))
\end{align*}
and
\begin{align*}
  &\H^n(\A(S),\A(S)^*)\cong\H^n(\ell^1(S),\ell^1(S)^*)\cong\\
 &\H^n(\A(S)^\#,(\A(S)^\#)^*)\cong\H^n(\ell^1(S)^\#,(\ell^1(S)^\#)^*)\cong\\
 &\H^n(\ell^1(G),\ell^1(G)^*)
\end{align*}
for $n\geq0$.
\end{Corollary}
\begin{proof}
The proof for homology and cohomology are identical.  Since the reduced
module of $\ell^1(S)$ is $\A(S)$, in both cases the first isomorphism
follows from \ref{reduced}.  The next two isomorphisms follow from
H-unitality, cf. \ref{hunital}. Finally the last isomorphism
is a consequence of \ref{main} since $\Phi(\A(S))=\ell^1(G)$.
\end{proof}

Recall that a Banach algebra $\A$ is \emph{weakly amenable} if
$\H^1(\A,\A^*)=\{0\}$.

\begin{Corollary} The algebras $\ell^1(S)^\#$, $\ell^1(S)$, $\A(S)^\#$,
  and $\A(S)$ are all weakly amenable.
\end{Corollary}
\begin{proof}
The Banach algebra $\ell^1(G)$ is weakly amenable \cite{[BEJ2]}.
\end{proof}

Recall that a Banach algebra $\A$ is \emph{biprojective} if
multiplication $\Pi\colon \A\pot\A\to\A$ has a right inverse in
$\amoda$, and is \emph{biflat} if the dual of multiplication
$\Pi^*\colon\A^*\to(\A\pot\A)^*$ has a left inverse in $\amoda$.

\begin{Corollary}\label{biflat}
$\A(S)$ is biflat if and only if $G$ is amenable.
\end{Corollary}
\begin{proof} By \cite[Theorem 5.8.(i)]{[Sel]} a Banach algebra $\A$
  is biflat if and only if it is self-induced and $\H^1(\A,X^*)=\{0\}$
  for all induced modules $X$. As $\ell^1(G)$, being unital, is biflat
  if and only if it is amenable, the result follows from
  \cite{[BEJ1]}.
\end{proof}

The corresponding result for biprojectivity is not immediate from
Morita theory as a description of biprojectivity in terms of Hochschild
cohomology involves non-induced modules. But we
can give a direct proof of

\begin{Theorem}
$\A(S)$ is biprojective if and only if $G$ is finite.
\end{Theorem}
\begin{proof} Assume that $\vert G\vert<\infty$. Choose an idempotent
  $e=(i,p_{\lambda i}^{-1},\lambda)$ and define
  $\rho\colon\A(S)\to\A(S)\pot\A(S)$ by
$$
\rho((j,g,\mu))=\frac1{\vert G\vert}\sum_{h\in G}(j,ghp_{\lambda
  i}^{-1},\lambda)\ot(i,h^{-1},\mu),\quad (j,g,\mu)\in I\times
G\times\Lambda.
$$
Then clearly $\Pi\circ \rho=\bold1$. One checks, as in the proof of
biprojectivity of group algebras over finite groups, that
$$
(j,g,\mu)\rho((j',g',\mu'))=\rho((j,g,\mu)(j',g',\mu'))=\rho((j,g,\mu))(j',g',\mu')
$$
for all $ (j,g,\mu),\,(j',g',\mu')\in I\times G\times\Lambda$, so that
$\rho$ is a bimodule homomorphism.

Conversely, suppose that $\A(S)$ is biprojective, and let
$\A(S)\to\A(S)\pot\A(S)$ be a splitting of multiplication. Consider
its restriction $\rho\colon e\A(S)e\to e\A(S)\pot\A(S)e$. Since $\rho$
is a bimodule homomorphism we have
$$
\rho(eae)=eae\rho(e)=\rho(e)eae,\quad a\in\A(S).
$$
Using the decomposition
$$
e\A(S)\pot\A(S)e=\bigoplus_{j,\mu}\ell^1(_iS_\mu\times{}_jS_\lambda)
$$
as a direct sum of $e\A(S)e$ bimodules, we can write
$$
\rho(e)=\sum_{j,\mu}{}\tau_{\mu j},\quad
\tau_{\mu j}\in\ell^1({}_iS_\mu\times{}_jS_\lambda).
$$
It follows that
$$
eae\tau_{\mu j}=\tau_{\mu j} eae
$$
for all $a\in\A(S),\, j\in I,\,\mu\in\Lambda$.

In the remainder of the proof it will be convenient to use the
multiplication on a projective tensor product of Banach algebras given
by $a\ot b\cdot a'\ot b':=aa'\ot b'b$.

For each $j\in I,\;\mu\in\Lambda$ choose $f_{\mu j}\in {}_jS_\lambda$ and
$e_{\mu j}\in{}_iS_\mu$ so that
$$
\Pi(\tau_{\mu j}\cdot f_{\mu j}\ot e_{\mu j})=\Pi(\tau_{\mu j}).
$$
This is clearly possible: If $p_{\mu j}=\nul$ choose
$e_{\mu j}$ and $f_{\mu j}$ arbitrarily. If  $p_{\mu j}\neq\nul$
choose
$f_{\mu j}=(j,p_{\mu j}^{-1},\lambda)$ and $e_{\mu j}=(i,p_{\lambda i}^{-1},\mu)$.

Now put
$$
\Delta=\sum_{j,\mu}\tau_{\mu j}\cdot f_{\mu j}\ot e_{\mu j}.
$$
Then $\Delta\in e\A(S)e\pot e\A(S)e$ and
\begin{align*}
\Pi(\Delta)&=e\\
eae \Delta&=\Delta eae,\quad a\in\A(S)
\end{align*}
so that $e\A(S)e$ has a diagonal and therefore is biprojective. Since
$e\A(S)e\cong\ell^1(G)$, the group $G$ must be finite.
\end{proof}

{\em Acknowledgement.} The first author gratefully acknowledges the
support of NSERC of Canada, and the third author thanks Universit\'e
Laval for its kind hospitality while part of this paper was being
written. Further work was done during the 19$^\mathrm{th}$
International Conference on Banach Algebras, held at B\c{e}dlewo,
14--24 July, 2009. All three authors gratefully acknowledge the
support for the meeting by the Polish Academy of Sciences, the
European Science Foundation under the ESF-EMS-ERCOM partnership, and
the Faculty of Mathematics and Computer Science of Adam Mickiewicz
University of Pozna\'n.

\bigskip

\end{document}